\documentclass[10pt]{amsart} 
\usepackage{amsfonts,amsmath,amssymb,verbatim,amsthm,amscd}

\newtheorem {thm}{Theorem}
\newtheorem {lem}[thm]{Lemma}
\newtheorem {cor}[thm]{Corollary}

\newcommand{\mit}{}

\newcommand{\C}{\mathbb C}

\newcommand{\Z}{\mathbb Z}

\newcommand{\T}{\mathbb T}

\newcommand{\be}{\begin{equation}}
\newcommand{\ee}{\end{equation}}
\newcommand{\ba}{\begin{eqnarray}}
\newcommand{\ea}{\end{eqnarray}}

\newcommand{\bb}{\begin{thebibliography}}
\newcommand{\eb}{\end{thebibliography}}

\newcommand{\eg}{\Gamma_{\! p,q}}

\begin{document}

\title[Elliptic hypergeometric integrals]
{\bf Determinants of elliptic\\  hypergeometric integrals}

\author{E. M. Rains}
\address{Department of Mathematics, Caltech, Pasadena, CA 91125, USA}

\author{V. P. Spiridonov}
\address{Bogoliubov Laboratory of Theoretical Physics,
JINR, Dubna, Moscow Region 141980, Russia}

\thanks{To be published in the Russian journal {\em Funct. Analysis and its Appl.}}

\begin{abstract}
We start from an interpretation of the $BC_2$-symmetric ``Type I''
(elliptic Dixon) elliptic hypergeometric integral evaluation as a formula
for a Casoratian of the elliptic hypergeometric equation, and give an
extension to higher-dimensional integrals and higher-order hypergeometric
functions.  This allows us to prove the corresponding elliptic beta
integral and transformation formula in a new way, by proving both sides
satisfy the same difference equations, and that the difference equations
satisfy a Galois-theoretical condition ensuring uniqueness of
simultaneous solution.
\end{abstract}

\maketitle

\section{Introduction}
Plain hypergeometric functions and their $q$-analogues
are widely used in mathematics and mathematical physics.
They can be defined either as infinite series or contour integrals,
more or less on an equal footing \cite{aar}.
As to the recently discovered elliptic hypergeometric functions,
the situation with them is different---their general instances are
defined only via integral representations. The general concept of elliptic
hypergeometric integrals was introduced in \cite{spi:theta}.
Elliptic beta integrals \cite{die-spi:selberg,
rai:trans,spi:beta,spi:theta,spi-war:inversions} are the
simplest representatives of integrals of such type.
In the univariate setting there is only one elliptic beta integral
\cite{spi:beta}, presently the top level known
generalization of the Euler beta integral.
In the multivariable case such integrals are grouped in three classes.

The $n$-dimensional type I elliptic beta integrals contain $2n+3$ free
parameters, and there are two known general methods of proving them
\cite{rai:trans,spi:short}. The type II integrals have a smaller number of
parameters, and admit a straightforward derivation from the type I integrals
\cite{die-spi:selberg,spi:theta}. Both types of these exactly
computable integrals admit higher-order extensions with more parameters,
with associated transformation laws \cite{rai:trans}. For the $BC_n$ root system,
these are elliptic analogues of an integral due to Dixon \cite{Dixon} (type I)
and of Selberg's famous integral \cite{aar} (type II, with $5+1$ parameters),
respectively.
Multiple elliptic beta integrals of the third class \cite{spi:theta}
can be represented as determinants of univariate integrals, which, in turn,
reduce to computable theta function determinants.

In the present paper, we follow up on the observation implicit in
\cite{rai:trans} that the elliptic Dixon integrals can be expressed as
determinants of univariate elliptic hypergeometric integrals, higher-order
analogues of the elliptic beta integral.  This allows us to give a new
proof of the corresponding evaluation formula and of the related
transformation formula established first in \cite{rai:trans}. The
evaluation result lifts Varchenko's determinant of univariate plain
hypergeometric integrals \cite{var1} and the Aomoto-Ito determinant
\cite{ai3} to the elliptic level and enriches the list of known computable
determinants compiled in \cite{kra:advanced}.

We use the following notation. The key infinite product is
$$
(z;p)_\infty :=\prod_{k=0}^\infty(1-zp^k),
$$
where $|p|<1$ and $z\in\C$. The elliptic theta function has the form:
$$
\theta_p(z):=(z;p)_\infty (z^{-1};p)_\infty,
$$
where $z\in\C^*$. It obeys the properties
$$
\theta_p(pz)=\theta_p(z^{-1})=-z^{-1}\theta_p(z)
$$
and $\theta_p(z)=0$ for $z=p^\Z$. We follow the standard useful convention that
$$
\theta_p(a_1,\ldots,a_m):=\prod_{k=1}^m\theta_p(a_k),
\qquad
\theta_p(tz^{\pm1}):=\theta_p(tz)\theta_p(tz^{-1}),
$$
and say that a meromorphic function $f(z)$ is
$p$-elliptic if $f(pz)=f(z)$. The simplest nonconstant $p$-elliptic function
(of the second order) has the form $\theta_p(az,bz)/
\theta_p(cz,dz)$, where $ab=cd$.

The standard elliptic gamma function, depending on two complex bases
$p$ and $q$ lying in the unit disc, $|p|,|q|<1$, has the form:
$$
\eg(z)=\prod_{j,k=0}^\infty\frac{1-z^{-1}p^{j+1}q^{k+1}}{1-zp^jq^k},
$$
where $z\in\C^*$. It obeys the properties $\eg(z)=\Gamma_{\! q,p}(z)$,
$$
\eg(qz)=\theta_p(z)\eg(z),\quad \eg(pz)=\theta_q(z)\eg(z),
$$
and has zeros at $z=p^{\Z_{>0}}q^{\Z_{>0}} $ and poles at
$z=p^{\Z_{\leq0}}q^{\Z_{\leq0}}$. The reflection formula has the form
$\eg(a)\eg(b)=1, \, ab=pq$; we set also
\begin{eqnarray*}
&& \eg(a_1,\ldots,a_m):=\prod_{k=1}^m\eg(a_k),\qquad
\eg(tz^{\pm1}):=\eg(tz)\eg(tz^{-1}),\quad
\\ &&
\eg(tz_1^{\pm1}z_2^{\pm1})=\eg(tz_1z_2)\eg(tz_1z_2^{-1})
\eg(tz_1^{-1}z_2)\eg(tz_1^{-1}z_2^{-1}).
\end{eqnarray*}

\section{The elliptic hypergeometric equation}

The following elliptic analogue of the Gauss hypergeometric function
was introduced in \cite{spi:theta,spi:thesis}
\begin{equation}
V(\underline{t};p,q)=\kappa\int_\T\frac{\prod_{j=1}^8\eg(t_jz^{\pm 1})}
{\eg(z^{\pm2})}\frac{dz}{2\pi\sqrt{-1}z},
\label{ehf}\end{equation}
where $\kappa=(p;p)_\infty(q;q)_\infty/2$ and $\T$ is the positively oriented
unit circle. The parameters $t_j$ are restricted by the balancing condition
$\prod_{j=1}^8t_j=(pq)^2$ and the inequalities $|t_j|<1,\, j=1,\ldots,8$.
The $V$-function can be meromorphically continued to all $t_j\in\C^*$
preserving the balancing condition. For $t_7t_8=pq$ (and other similar
restrictions), it reduces to the elliptic beta integral \cite{spi:beta}
\begin{equation}
\kappa\int_\T\frac{\prod_{j=1}^6\eg(t_jz^{\pm 1})}
{\eg(z^{\pm2})}\frac{dz}{2\pi\sqrt{-1}z}=\prod_{1\leq j<k\leq 6}\eg(t_jt_k).
\label{e-beta}\end{equation}

The addition formula for elliptic theta functions written in the form
$$
t_3\theta_p(t_2t_3^{\pm1},t_1z^{\pm1})+t_1\theta_p(t_3t_1^{\pm1},t_2z^{\pm1})
+t_2\theta_p(t_1t_2^{\pm1},t_3z^{\pm1})=0
$$
yields the following contiguity relation, via a corresponding relation for
the integrands:
\begin{equation}
\frac{t_1V(qt_1)}{\theta_p(t_1t_2^{\pm1},t_1t_3^{\pm 1})}
+\frac{t_2V(qt_2)}{\theta_p(t_2t_1^{\pm1},t_2t_3^{\pm 1})}
+\frac{t_3V(qt_3)}{\theta_p(t_3t_1^{\pm1},t_3t_2^{\pm 1})} = 0,
\label{cont-1}\end{equation}
where $V(qt_j)$ denotes the $V(\underline{t};p,q)$-function
with the parameter $t_j$ replaced by $qt_j$ (with the balancing condition
being $\prod_{j=1}^8t_j=p^2q$).
Bailey-type symmetry transformations for the $V$-function \cite{rai:trans,spi:theta}
give to \eqref{cont-1} different forms, including in particular:
\begin{equation}
\frac{\prod_{j=4}^8\theta_p\left(t_1t_j/q\right)V(t_1/q)}
{t_1\theta_p(t_2/t_1,t_3/t_1)}
+\frac{\prod_{j=4}^8\theta_p\left(t_2t_j/q\right)V(t_2/q)}
{t_2\theta_p(t_1/t_2,t_3/t_2)}
+\frac{\prod_{j=4}^8\theta_p\left(t_3t_j/q\right)V(t_3/q)}
{t_3\theta_p(t_1/t_3,t_2/t_3)} = 0,
\label{cont-3}\end{equation}
where $\prod_{j=1}^8t_j=p^2q^3$.  In combination with \eqref{cont-1}, this
yields the elliptic hypergeometric equation \cite{spi:thesis}:
\begin{eqnarray}\label{eheq}
&& \makebox[-2em]{}
\mathcal{A}(t_1,t_2,\ldots,t_8,q;p)\Big(U(qt_1,q^{-1}t_2;q,p)-U(\underline{t};q,p)\Big)
\\ &&
+\mathcal{A}(t_2,t_1,\ldots,t_8,q;p)\Big(U(q^{-1}t_1,qt_2,;q,p)-U(\underline{t};q,p)\Big)
+ U(\underline{t};q,p)=0,
\nonumber\end{eqnarray}
where we have denoted
\begin{equation}
 \mathcal{A}(t_1,\ldots, t_8,q;p):=\frac{\theta_p(t_1/qt_3,t_3t_1,t_3/t_1)}
                 {\theta_p(t_1/t_2,t_2/qt_1,t_1t_2/q)}
\prod_{k=4}^8\frac{\theta_p(t_2t_k/q)}{\theta_p(t_3t_k)}
\end{equation}
and
$$
U(\underline{t}; q,p):=\frac{V(\underline{t};q,p)}
{\prod_{k=1}^2\mit\eg(t_kt_3^{\pm 1})}.
$$
The potential $\mathcal{A}(t_1,\ldots, t_8,q;p)$
is a $p$-elliptic function of parameters $t_1,\ldots,t_8$, one of which should be
counted as a dependent variable through the balancing condition.

We set $t_1=(pq)^2/t_2\cdots t_8$ and perform the shift $t_2\to pt_2$
(so that $t_1\to t_1/p$). Since the function $\mathcal{A}$ is $p$-elliptic in
all parameters, we have $\mathcal{A}(p^{-1}t_1,pt_2,\ldots)=
\mathcal{A}(t_1,t_2,\ldots)$. The function $U(p^{-1}t_1,pt_2)$ defines
therefore an independent solution of the elliptic hypergeometric equation.
Let us compute the Casoratian of these two solutions (i.e., a discrete
version of the Wronskian). For this, we multiply the above equation by
$U(p^{-1}t_1,pt_2)$, the equation
\ba\nonumber
&& \mathcal{A}(t_1,t_2,\ldots t_8,q;p)\Big(U(p^{-1}qt_1,pq^{-1}t_2)-U(p^{-1}t_1,pt_2)\Big)
\\ && \makebox[2em]{}
+\mathcal{A}(t_2,t_1,t_3,\ldots,q;p)\Big(U(p^{-1}q^{-1}t_1,pqt_2)-U(p^{-1}t_1,pt_2)\Big)
+ U(p^{-1}t_1,pt_2)=0
\nonumber\ea
by $U(t_1,t_2)$, subtract them and obtain
\begin{equation}\label{cas-eqn}
\mathcal{A}(t_1,t_2,\ldots t_8,q;p){D}(p^{-1}t_1,q^{-1}t_2)
=\mathcal{A}(t_2,t_1,t_3,\ldots,q;p){D}(p^{-1}q^{-1}t_1,t_2),
\end{equation}
where
$$
{D}(t_1,t_2)=U(qpt_1,t_2)U(t_1,pqt_2)-U(qt_1,pt_2)U(pt_1,qt_2)
$$
is the needed Casoratian. It is symmetric in $p$ and $q$,
which is an important property. The expression
$$
{D}(t_1,t_2)=\frac{V(pqt_1,t_2)V(t_1,pqt_2)-t_1^{-2}t_2^{-2}
\; V(qt_1,pt_2)V(pt_1,qt_2)}
{\prod_{k=1}^2 \eg(t_kt_3^{\pm 1},pqt_k t_3^{\pm 1})}
$$
can obviously be interpreted as the determinant of a particular $2\times 2$ matrix
whose elements are expressed via the $V$-function.

Since $t_1$ is a dependent variable, this is actually a first order
difference equation in $t_2$. After scaling $t_1\to pt_1,\;
t_2\to qt_2$ (so that $t_1=pq/\prod_{j=2}^8t_j$),
we obtain the following equation for $f(t_2):={D}(t_1,t_2)$:
\begin{eqnarray*}
&& f(qt_2)=\frac{\mathcal{A}(pt_1,qt_2,t_3,\ldots,q;p)}
{\mathcal{A}(qt_2,pt_1,t_3,\ldots,q;p)}f(t_2)
\\ && \makebox[2em]{}
=-\frac{t_1}{qt_2}\frac{\theta_p(t_1/q^2t_2,t_1/qt_3,t_1^{-1}t_3^{\pm 1})}
            {\theta_p(t_2/t_1,t_2/t_3, q^{-1}t_2^{-1}t_3^{\pm 1})}
\prod_{k=4}^8\frac{\theta_p(t_2t_k)}{\theta_p(t_1t_k/q)}f(t_2),
\end{eqnarray*}
which yields
$$
{D}(t_1,t_2)=C(t_2)\;\frac{\prod_{k=3}^8\eg(t_1t_k,t_2t_k) }
{\eg(t_1/t_2,t_2/t_1)}\prod_{k=1}^2\frac{\eg(t_k^{-1}t_3^{\pm 1})}
{\eg(t_kt_3^{\pm 1})},
$$
where $C(qt_2)=C(t_2)$. We can repeat the whole consideration with
permuted $p$ and $q$ and obtain $C(pt_2)=C(t_2)$. This means
(for incommensurate $p$ and $q$) that $C$ does not depend on $t_2$, but it
may depend on other parameters $t_3,\ldots,t_8$.
To compute $C$, we apply the residue calculus. For this we take the parameter
$t_3$ from inside the unit circle to its outside and impose the constraints
$|t_3|>1>|qt_3|,|pt_3|$. Then we deform the contour of integration $\T$ entering
the definition of $V(\underline{t})$ to the contour $\T_{def}$ deformed in such a
way that no poles are crossed during such a change of $t_3$. The Cauchy theorem
leads to
$$
V(\underline{t}):=V_{\T_{def}}(\underline{t})=V_\T(\underline{t})
+\frac{\prod_{j=1,\neq 3}^8\eg(t_jt_3^{\pm 1})}{\eg(t_3^{-2})}.
$$
We take then the limit $t_2\to 1/t_3$ and find the value of $C$ through the limit for
ratios of the left and right-hand sides of the above equality
$$
C=\lim_{t_2t_3\to 1}
\frac{V(pqt_1,t_2)V(t_1,pqt_2)-t_1^{-2}t_2^{-2}
\; V(qt_1,pt_2)V(pt_1,qt_2)}
{\prod_{k=3}^8\eg(t_1t_k,t_2t_k)} \eg(t_1/t_2,t_2/t_1).
$$
For $t_2\to1/t_3$, the function $V(t_1,pqt_2)$ reduces to the
elliptic beta integral, the residues of $V(pqt_1,t_2)$ blow up with
$V_\T(pqt_1,t_2)$ remaining finite, the residues of $V(pt_1,qt_2)$ and
$V(qt_1,pt_2)$ remain finite as well as the functions $V_\T(pt_1,qt_2)$ and
$V_\T(qt_1,pt_2)$. As a result, only the first term of our Casoratian
survives and yields
$$
C=\frac{\prod_{3\leq j<k\leq 8}\eg(t_jt_k)}{\eg(t_1^{-1}t_2^{-1})}.
$$
We obtain thus the formula
\begin{equation}
V(pqt_1,t_2)V(t_1,pqt_2)-t_1^{-2}t_2^{-2} V(qt_1,pt_2)V(pt_1,qt_2)
=\frac{\prod_{1\leq j<k\leq8}\eg(t_jt_k)}{\eg(t_1^{\pm1}t_2^{\pm1})}.
\label{V-det}\end{equation}

The described solution of the elliptic hypergeometric equation is defined for $|q|<1$,
although equation \eqref{eheq} itself
does not demand such a condition. It can be verified that
$$
\mathcal{A}\left(\frac{p^{1/2}}{t_1}, \ldots,\frac{p^{1/2}}{t_8},q;p\right)
=\mathcal{A}\left(t_1,\ldots,t_8,q^{-1};p\right).
$$
The transformation $t_j\to p^{1/2}/t_j$, $j=1,\ldots, 8,$
maps therefore the elliptic hypergeometric equation to itself with the base
change $q\to q^{-1}$. Equivalently, the same inversion $q\to 1/q$ occurs
after the transformation $t_j\to p^{a_j}/t_j$ with
integer $a_i$ such that $\sum_{j=1}^8a_j=4$.
As a result, we
obtain the following solution of the elliptic hypergeometric equation
in the regime $|q|>1$
\begin{eqnarray}
U(\underline{t};q,p)=\frac{V(p^{1/2}/t_1,\ldots,p^{1/2}/t_8;q^{-1},p)}
{\prod_{k=1}^2\Gamma_{p,q^{-1}}(p/t_kt_3,t_3/t_k) }.
\label{q>1}\end{eqnarray}
As to the unit circle case $|q|=1$, the corresponding solution of
equation \eqref{eheq} can
be obtained with the help of the modified elliptic gamma function
or the modular transformation \cite{spi:thesis}.

\section{A characterization theorem for the $V$-function}

We would like now to present contiguous relations for
the $V$-function (and, so, equation \eqref{eheq}) in a $2\times 2$
matrix form. For that we introduce the function
$$
W(t_1,\ldots,t_8;z):=\frac{V(t_1,\ldots,t_8;p,q)}{\prod_{j=1}^8\eg(t_jz^{\pm1})},
$$
where $z$ is some auxiliary variable. Replacing parameters $t_{1,2,3}$ by
$t_{4,7,8}$ and the $V$-function by $W$ in \eqref{cont-1}, we obtain
after shifting $t_3\to qt_3$
\begin{equation}
W(qt_3,qt_7)=\alpha(\underline{t};z)W(qt_3,qt_4)
+\beta(\underline{t};z)W(qt_3,qt_8),
\label{c-1}\end{equation}
where $W(qt_j,qt_k)$ means the $W(\underline{t};z)$-function with
respective parameters $t_j$ and $t_k$ replaced by $qt_j$ and $qt_k$, and
$$
\alpha(\underline{t};z)=\frac{\theta_p(t_4z^{\pm1},t_7t_8^{\pm1})}
{\theta_p(t_7z^{\pm1},t_4t_8^{\pm1})},
\qquad
\beta(\underline{t};z)=\frac{\theta_p(t_8z^{\pm1},t_7t_4^{\pm1})}
{\theta_p(t_7z^{\pm1},t_8t_4^{\pm1})}
$$
are $p$-elliptic functions of all variables (including $z$).

Replacing now $t_{1,2,3}$ by $qt_{1,2,3}$ in  \eqref{cont-3},
and then permuting $t_1$ and $t_2$ with $t_4$ and $t_8$, we obtain
\begin{equation}
W(qt_3,qt_4)=\gamma(\underline{t};z)W(qt_3,qt_8)+\delta(\underline{t};z)W(qt_4,qt_8),
\label{c-2}\end{equation}
where
\begin{eqnarray}\nonumber
&& \gamma(\underline{t};z)=\frac{\theta_p(t_8z^{\pm1},t_3t_8^{-1})}
{\theta_p(t_4z^{\pm1},t_3t_4^{-1})}\prod_{j=1,2,5,6,7}
\frac{\theta_p(t_4t_j)}{\theta_p(t_8t_j)},\qquad
\\ &&
\delta(\underline{t};z)=\frac{\theta_p(t_8z^{\pm1},t_4t_8^{-1})}
{\theta_p(t_3z^{\pm1},t_4t_3^{-1})}\prod_{j=1,2,5,6,7}
\frac{\theta_p(t_3t_j)}{\theta_p(t_8t_j)}
\label{c-2-coeff}\end{eqnarray}
are, again, $p$-elliptic functions of the parameters.
Eliminating $W(qt_3,qt_4)$ from \eqref{c-1} and \eqref{c-2}, we obtain
the relation
\begin{equation}
W(qt_3,qt_7)=(\alpha(\underline{t};z)\gamma(\underline{t};z)
+\beta(\underline{t};z))W(qt_3,qt_8)+\alpha(\underline{t};z)
\delta(\underline{t};z) W(qt_4,qt_8).
\label{c-3}\end{equation}
We define now the matrices
\begin{eqnarray}\label{M-mat}
&& M(t_1,t_2;t_3,t_4;t_5,t_6,t_7,t_8):=\left(\begin{array}{cc}
W(pt_1,qt_3) & W(pt_2,qt_3) \cr
W(pt_1,qt_4) & W(pt_2,qt_4) \cr
\end{array}\right), \quad
\\ &&
A(t_1,t_2;t_3,t_4;t_5,t_6,t_7,t_8):=\left(\begin{array}{cc}
A_{11} & A_{12} \cr
A_{21}& A_{22} \cr
\end{array}\right),
\label{A-mat}\end{eqnarray}
where
\begin{eqnarray*}
&& A_{11}(t_1,t_2;t_3,t_4;t_5,\ldots,t_8;z;p,q)
=\alpha(pt_1,q^{-1}t_8)\gamma(pt_1,q^{-1}t_8)+\beta(pt_1,q^{-1}t_8),
\\
&& A_{12}(t_1,t_2;t_3,t_4;t_5,\ldots,t_8;z;p,q)
=\alpha(pt_1,q^{-1}t_8)\delta(pt_1,q^{-1}t_8),
\\
&& A_{21}(t_1,t_2;t_3,t_4;t_5,\ldots,t_8;z;p,q)
=A_{12}(t_1,t_2;t_4,t_3;t_5,\ldots,t_8;z;p,q),
\\
&& A_{22}(t_1,t_2;t_3,t_4;t_5,\ldots,t_8;z;p,q)
=A_{11}(t_1,t_2;t_4,t_3;t_5,\ldots,t_8;z;p,q)
\end{eqnarray*}
are $p$-elliptic functions. In particular,
$$
A_{ij}(p^{-1}t_1,pt_2;t_3,t_4;t_5,\ldots,t_8;z;p,q)
=A_{ij}(t_1,t_2;t_3,t_4;t_5,\ldots,t_8;z;p,q).
$$
After replacements $t_7\to t_7x,\, t_8\to t_8 x^{-1}$,
equations \eqref{c-3} and its partner obtained after permuting
$t_3$ with $t_4$ are rewritten
as a linear first order matrix $q$-difference equation
\begin{equation}
M(qx)=A(x)\, M(x),
\label{A-eq}\end{equation}
where we indicate only $x$-dependence.

After the permutations $p\leftrightarrow q, t_1\leftrightarrow t_3,
t_2\leftrightarrow t_4$, the matrix $M$ is transformed to its
transpose $M^T$. Equation \eqref{A-eq} gets therefore transformed to
\begin{equation}
M(px)=M(x)\,B(x),
\label{B-eq}\end{equation}
where
\begin{equation}
B(x):=B(t_1,t_2;t_3,t_4;t_5,t_6,t_7x,t_8x^{-1}):=\left(\begin{array}{cc}
B_{11} & B_{12} \cr
B_{21} & B_{22} \cr
\end{array}\right),
\label{B-mat}\end{equation}
\begin{eqnarray*}
&& B_{11}=A_{11}(t_3,t_4;t_1,t_2;t_5,t_6,t_7x,t_8x^{-1};z;q,p),
\\
&& B_{12}=A_{21}(t_3,t_4;t_1,t_2;t_5,t_6,t_7x,t_8x^{-1};z;q,p),
\\
&& B_{21}=A_{12}(t_3,t_4;t_1,t_2;t_5,t_6,t_7x,t_8x^{-1};z;q,p),
\\
&& B_{22}=A_{22}(t_3,t_4;t_1,t_2;t_5,t_6,t_7x,t_8x^{-1};z;q,p).
\end{eqnarray*}

In principle, from the existence of first equation \eqref{A-eq}, it follows
that there exists {\em some} $p$-difference equation of the form
\eqref{B-eq} with $q$-elliptic coefficients  \cite{Etingof}.
Indeed, we can simply take $B(x):=M(x)^{-1} M(px) $ and see that
\begin{eqnarray*} &&
B(qx) = M(x)^{-1}A(x)^{-1}A(px)M(px) = M(x)^{-1}M(px) = B(x).
\end{eqnarray*}
However, this $B$-matrix is not unique. We let $g(x)$ denote
a matrix satisfying $g(qx)=g(x).$
Equation \eqref{A-eq} does not change after the replacement $M\to Mg$,
but the $B$-matrix gets changed to $ B\to g(x)^{-1}B\, g(px) $
showing a functional freedom in the definition of this matrix.

We suppose now, that $M'$ is another meromorphic solution of
equations \eqref{A-eq} and \eqref{B-eq}. The matrix $N = M'M^{-1} $
satisfies then the difference equations
$$
N(px) = N(x), \quad
N(qx) = A(x) N(x) A(x)^{-1}.
$$
The first equation states that $N$ has $p$-elliptic entries, and the second
one is a $q$-difference equation with the $p$-elliptic coefficients.
The normalization chosen above works for $q$- and $p$-shifts of all parameters,
not just those of $t_7$ and $t_8$, because of the permutational symmetry.

\begin{thm}\label{V-char}
For incommensurate $p$ and $q$, $N$ is a constant multiple of the identity.
\end{thm}

The proof of this Theorem is given in a more general context in the last
section.  This statement simply means that the
non-trivial $p$-elliptic
functions (the matrix elements of $N$) cannot satisfy this $q$-difference
equation with $p$-elliptic coefficients.

With the help of three pairwise incommensurate quasiperiods $\omega_{1,2,3}$
and the parametrization
\begin{eqnarray*}
&& q= e^{2\pi i\frac{\omega_1}{\omega_2}}, \quad
p=e^{2\pi i\frac{\omega_3}{\omega_2}},
\end{eqnarray*}
we can convert $q$- and $p$-difference equations into linear finite
difference equations. For this it is sufficient to pass to parameters $g_j$
introduced as $t_j=e^{2\pi ig_j/\omega_2},\; j=1,\ldots,8,$ and $x=e^{2\pi
  iu/\omega_2}.$ As a function of $u$, the $V$-function is thus
characterized as a unique (up to a constant independent of $u$) solution of
5 linear fininite difference equations: two difference equation in
\eqref{A-eq} working with the shifts $u\to u+\omega_1$, their partners in
\eqref{B-eq} working with the shifts $u\to u+\omega_3$, and the condition
of periodicity under the shifts $u\to u+\omega_2$, equivalent to the
analiticity condition in $x\in\C^*$. In order to characterize the proper function
$V(\underline{t};q,p)$ as a function of $t_1,\ldots,t_8$ up to a constant,
we simply need to adjoin the permutational symmetry group $S_8$.

\section{Determinant representation of elliptic Dixon integrals}

The elliptic Dixon integrals (a.k.a. Type I integrals with $BC_n$ symmetry)
have the form:
$$
I_n^{(m)}(t_1,\ldots,t_{2n+2m+4})=\kappa_n\int_{\T^n}
\prod_{1\leq i<j\leq n} \frac{1}{\eg(z_i^{\pm 1}z_j^{\pm 1})}
\prod_{j=1}^n\frac{\prod_{i=1}^{2n+2m+4}\eg(t_iz_j^{\pm 1})}
{\eg(z_j^{\pm 2})}\frac{dz_j}{2\pi\sqrt{-1}z_j},
$$
where $|t_j|<1$,
$$
\prod_{j=1}^{2n+2m+4}t_j=(pq)^{m+1},\qquad
\kappa_n=\frac{(p;p)_\infty^n(q;q)_\infty^n}{2^n n!}.
$$
By convention, when $n=0$, $I^{(m)}_0:=1$; when $n=1$, the resulting
univariate integral is a higher-order version of the elliptic beta
integral.

The following transformation identity has been proved in \cite{rai:trans}.
\begin{thm}\label{tr-thm}
The integrals $I^{(m)}_n$ satisfy the relation
\begin{equation}
I_n^{(m)}(t_1,\ldots,t_{2n+2m+4})=\prod_{1\leq r<s\leq 2n+2m+4}\eg(t_rt_s)\;
I_m^{(n)}\left(\frac{\sqrt{pq}}{t_1},\ldots,\frac{\sqrt{pq}}{t_{2n+2m+4}}\right).
\label{trafo}\end{equation}
\end{thm}
In particular, when $m=0$, the integral on the right-hand side is
0-dimensional, and one obtains an explicit evaluation of the left-hand side.

One of the objectives of the present work is to give a more elementary
proof of this transformation, by showing that both sides satisfy the same
family of difference equations and initial conditions.  We will also obtain
a new proof of the special case $m=0$.

The key idea is that $I^{(m)}_n$ can be written as a determinant of
integrals of the form $I^{(n+m-1)}_1$; we will thus be able to use
difference equations for the univariate integrals to deduce difference
equations for $I^{(m)}_n$.

Due to the reflection identity $\eg(z)\eg(pq/z)=1$,
we can write the $I_n^{(m)}$ integrand's cross factors as
$$
\prod_{1\leq i<j\leq n} \frac{1}{\eg(z_i^{\pm 1}z_j^{\pm 1})}
=\prod_{1\le i<j\le n} \left( z_i^{-1}\theta_p(z_i z_j^{\pm 1})\;
z_i^{-1}\theta_q(z_i z_j^{\pm 1})\right).
$$
The point, then is that the antisymmetric factor $\prod_{1\le i<j\le n}
z_i^{-1}\theta_p(z_iz_j^{\pm 1})$ can be written as a determinant; more
precisely, we have the following elliptic analogue of the Cauchy determinant:
$$
\det_{1\le i,j\le n} \left(\frac{1}{a_i^{-1}\theta_p(a_iz_j^{\pm 1})}\right)
=
\frac{
\prod_{1\le i<j\le n} a_i^{-1}\theta_p(a_i a_j^{\pm 1})
\prod_{1\le i<j\le n} z_i^{-1}\theta_p(z_i z_j^{\pm 1})}
{(-1)^{n(n-1)/2}\prod_{1\le i,j\le n} a_i^{-1}\theta_p(a_i z_j^{\pm 1})}
$$
(which, as observed in \cite{rai:rec}, can be obtained from the usual
Cauchy determinant via a suitable substitution), using which we can bring
our integral to the form
\begin{eqnarray*} && \makebox[-1em]{}
I_n^{(m)}(t_1,\ldots,t_{2n+2m+4})=\frac{\kappa_n}
{\prod_{1\leq i<j\leq n}a_i^{-1}\theta_p(a_ia_j^{\pm 1})b_i^{-1}\theta_q(b_ib_j^{\pm 1})}
\\ && \makebox[0em]{} \times
\int_{\T^n}\prod_{j=1}^n\left(
\frac{\prod_{r=1}^{2n+2m+4}\eg(t_rz_j^{\pm 1})}{\eg(z_j^{\pm2})}
\prod_{k=1}^n b_k^{-1}\theta_p(b_kz_j^{\pm 1})\; b_k^{-1} \theta_q(b_kz_j^{\pm 1})
\frac{dz_j}{2\pi\sqrt{-1}z_j}\right)
\\ && \makebox[8em]{} \times
\det_{1\le i,j\le n} \phi_i(z_j) \det_{1\le i,j\le n} \psi_i(z_j),
\end{eqnarray*}
where $\phi_i(z_j)=a_i/\theta_p(a_iz_j^{\pm 1})$,
$\psi_i(z_j)=b_i/\theta_q(b_iz_j^{\pm 1})$.
But then, using the Heine identity:
$$
\frac{1}{n!}
\int
\det_{1\le i,j\le n} \phi_i(z_j)
\det_{1\le i,j\le n} \psi_i(z_j)
\prod_{1\le i\le n} d\mu(z_i)
= \det_{1\le i,j\le n}\int \phi_i(z)\psi_j(z) d\mu(z),
$$
we can write
\begin{eqnarray*} &&\makebox[-1em]{}
I_n^{(m)}(t_1,\ldots,t_{2n+2m+4})= \prod_{1\leq i<j\leq n}
\frac{1}{a_j\theta_p(a_ia_j^{\pm 1})b_j\theta_q(b_ib_j^{\pm 1})}
\\ && \makebox[-1em]{} \times
\det_{1\le i,j\le n}\left(\kappa\int_\T \frac{\prod_{r=1}^{2n+2m+4}
\eg(t_r z^{\pm 1})}{\eg(z^{\pm2})}
\prod_{k\ne i} \theta_p(a_kz^{\pm 1})
\prod_{k\ne j} \theta_q(b_kz^{\pm 1})
\; \frac{dz}{2\pi\sqrt{-1}z} \right),
 \end{eqnarray*}
where $\kappa=(p;p)_\infty (q;q)_\infty/2$.  If we choose $a_i=t_i$,
$b_i=t_{n+i}$, $1\le i\le n$, then the entries of the above determinant are
of the form $I^{(m+n-1)}_1$.  To be precise, the $ij$ entry is
\[
T_q(t_i)^{-1} T_p(t_{n+j})^{-1}
I^{(m+n-1)}_1(qt_1,\dots,qt_n,pt_{n+1},\dots,pt_{2n},t_{2n+1},\dots,t_{2n+2m+4}),
\]
where $T_q(t_k)$ represents the $q$-shift operator,
$T_q(t_k)f(t_k)=f(qt_k)$. We can also let the sequences $a$, $b$ overlap,
at the cost of some slight extra complication. For instance,
for $n=2, m=0$ and $a_i=b_i=t_i,\; i=1,2,$ we obtain on the right-hand side
the Casoratian of $V$-functions \eqref{V-det}, which yields
$
I_2^{(0)}(t_1,\ldots,t_8)=\prod_{1\leq j<k\leq 8}\eg(t_jt_k).
$
This result hints that the integral $I_n^{(0)}$ is computable
in the closed form for arbitrary $n>2$.

Note that when $m=-1$, the balancing condition reads $t_1\cdots
t_{2n+4}=1$, and thus at least one of the entries of the determinant
cannot use the unit circle as its contour.  This is not, however, a serious
issue, since there always exists {\em some} valid choice of common contour,
and the Heine identity works regardless.

Now, the fact that the integral vanishes for $m=-1$ (a fact visible from
the explicit formula for $m=0$) implies that the rows of the above matrix
must be linearly dependent.  Conversely, if we can find a linear dependence
between integrals
\[
T_q(t_i)^{-1} T_p(t_{n+j})^{-1}
I^{(n-2)}_1(qt_1,\dots,qt_n,pt_{n+1},\dots,pt_{2n},t_{2n+1},t_{2n+2}),
\]
which is independent of $j$ (i.e., is $p$-elliptic in $t_{n+1},\dots,t_{2n}$),
that will imply vanishing for $m=-1$, which as shown in \cite{die-spi:selberg}
allows one to compute the integral for $m=0$.  This leads us to examine recurrence
relations for the univariate integrals.

In general, not only are the entries of the above matrix integrals
of the form $I^{(n+m-1)}_1$, but in fact the $k\times k$ minors
themselves are proportional to integrals of the form $I^{(n+m-k)}_k$.
The recurrence for univariate integrals implicit in the vanishing of
$I^{(-1)}_n$ gives rise to recurrences for higher $I^{(m)}_n$ in the
following way.

\begin{lem}\label{lem-kernel}
Let $M$ be a $n\times k$ matrix, and suppose the vector $v$ satisfies $vM=0$.
Then the $k\times k$ minors of $M$ satisfy the $(n-k+1)$-term relation
\[
\sum_{k\le i\le n} v_i \det_{l\in\{1,\dots,k-1,i\},l'\in\{1,\dots,k\}}(M_{ll'}) = 0.
\]
\end{lem}

\begin{proof}
This certainly holds, by linearity, if we were to extend the sum
down to $i=1$, but the additional terms all have repeated rows
in the minors, thus vanish.
\end{proof}

There are two main sources of recurrences for hypergeometric integrals.  The
first is recurrences of the integrands themselves (so long as the contour
conditions can be satisfied by a common contour, that is).

\begin{thm}
The integral $I^{(m)}_n$ satisfies the $(n+2)$-term recurrence
\begin{equation}
\sum_{1\le i\le n+2}
\frac{t_i}{\prod_{j=1,\, j\ne i}^{n+2} \theta_p(t_i t_j^{\pm 1})}
T_q(t_i) I^{(m)}_n(t_1,\dots,t_{2n+2m+4})
=0,\quad \prod_{j=1}^{2n+2m+4}t_j=(pq)^mp.
\label{rec-rel1}\end{equation}
\end{thm}

\begin{proof}
If we divide out by common factors of the integrand, this reduces to the
relation (see Lemma A.1 in \cite{die-spi:selberg} or Corollary 2.3 in \cite{rai:rec}):
\[
\sum_{1\le i\le n+2}
\frac{t_i}{\prod_{j=1,\, j\ne i}^{n+2} \theta_p(t_i t_j^{\pm 1})}
\prod_{1\le j\le n} \theta_p(t_i z_j^{\pm 1})
=
0.
\]
\end{proof}

This approach alone is insufficient to get the full system of difference
equations; for one thing, it is completely independent of the balancing
condition.  As in \cite{rai:rec}, the key is to multiply the integrand by
functions related to the difference operators of \cite{rai:trans}.

With this in mind, we consider the following function:
\begin{eqnarray*}
&& g^{(m)}(z;t_1,\dots,t_{m+2};v_1,\dots,v_{m+4})
\\ && \makebox[2em]{}
=  \frac{\prod_{1\le i\le m+4}\theta_p(v_i z)}
     {z\theta_p(z^2)\prod_{1\le i\le m+2}\theta_p((pq/t_i)z)}
+  \frac{\prod_{1\le i\le m+4}\theta_p(v_i/z)}
     {z^{-1}\theta_p(z^{-2})\prod_{1\le i\le m+2}\theta_p((pq/t_i)/z)}.
\end{eqnarray*}
By inspection, this is invariant under the change $z\mapsto 1/z$.
In addition, if
$\prod_{1\le i\le m+2}t_i$ $\prod_{1\le i\le m+4}v_i$ $ = (pq)^{m+1}q$,
then both terms are (meromorphic) $p$-theta functions with the same
multiplier, and thus
\[
g^{(m)}(pz;t_1,\dots,t_{m+2};v_1,\dots,v_{m+4})
=
pz^2 g^{(m)}(z;t_1,\dots,t_{m+2};v_1,\dots,v_{m+4}).
\]
Moreover, the apparent poles at $z=\pm 1,\pm\sqrt{p}$ must by symmetry have
even order, and thus $g^{(m)}$ must in fact be holomorphic at those points.
Thus $g^{(m)}$ has only simple poles at the points $(pq/t_i)^{\pm 1} p^\Z$;
it follows that
\[
g^{(m)}(z;t_1,\dots,t_{m+2};v_1,\dots,v_{m+4})
=
\sum_{1\le i\le m+2} \frac{\alpha_i}{\theta_p(pq z^{\pm 1}/t_i)}
\]
for suitable coefficients $\alpha_i$ which can be computed in the usual
way: clear the denominator and set $z = t_i/q$ to find
\[
\alpha_i =
\frac{q\prod_{1\le j\le m+4}\theta_p(v_j t_i/q)}
     {t_i\prod_{j=1,\ne i}^{m+2}\theta_p(t_j/t_i)}.
\]

The relevance of this function for our purposes is that it integrates to 0
against the $I^{(m)}_1$ density.  More precisely, we have
\[
\int_{|z|=1} g^{(m)}(z;t_1,\dots,t_{m+2};t_{m+3},\dots,t_{2m+6})
\frac{\prod_{1\le r\le 2m+6} \eg(t_r z^{\pm 1})}
     {\eg(z^{\pm 2})}
\frac{dz}{2\pi\sqrt{-1}z}
=
0,
\]
so long as $\prod_{1\le i\le 2m+6}t_i=(pq)^{m+1}q$ and
$|t_1|/q,\dots,|t_{m+2}|/q,|t_{m+3}|,\dots,|t_{2m+6}|<1$.
Indeed, by symmetry, both terms of $g^{(m)}$ have the same integral; if we
restrict to the first term (thus gaining a factor of 2), perform the
change of variables $z\mapsto q^{-1/2}/z$, and move the contour back to the
unit circle (which crosses over no poles), we obtain
\[
2
\int_{|z|=1}
\frac{q^{1/2} z}{\theta_p(z^2)}
\frac{
\prod_{1\le i\le m+2} \eg(q^{-1/2} t_i z^{\pm 1})
\prod_{1\le i\le m+4} \eg(q^{1/2} t_{m+2+i} z^{\pm 1})
}
{\eg(z^{\pm 2})}
\frac{dz}{2\pi\sqrt{-1}z}.
\]
But now the integrand is antisymmetric with respect to $z\mapsto z^{-1}$,
and therefore the integral vanishes.

Using the partial fraction decomposition of $g^{(m)}$, we thus obtain
a new recurrence.

\begin{lem}
If $t_1\cdots t_{2m+6} = (pq)^{m+1}q$, then we have the following
$(m+2)$-term recurrence for $I^{(m)}_1$:
\begin{equation}
\sum_{1\le k\le m+2}
\frac{\prod_{m+3\le i\le 2m+6} \theta_p(t_it_k/q)}
     {t_k\prod_{1\le i\le m+2;i\ne k} \theta_p(t_i/t_k)}
T_q(t_k)^{-1}
I^{(m)}_1(t_1,\dots,t_{2m+6}) =0.
\label{rec-rel}\end{equation}
\end{lem}

Applying the operator $T_p(t_{m+2+l})^{-1}$ to this equality,
we obtain
\begin{eqnarray}\label{zero-mode}
&& \sum_{1\leq k \leq m+2}v_kM_{kl}=0,\quad
v_k=\frac{\prod_{m+3\le i\le 2m+6} \theta_p(t_it_k/q)}
     {\prod_{1\le i\le m+2;i\ne k} \theta_p(t_i/t_k)},
\quad
\\ && \makebox[2em]{}
M_{kl}=T_q(t_k)^{-1}T_p(t_{m+2+l})^{-1}
I^{(m)}_1(t_1,\dots,t_{2m+6}).
\nonumber\end{eqnarray}

\begin{cor}
If $t_1\cdots t_{2n+2} = 1$, then $I^{(-1)}_n=0$.
\end{cor}

Indeed, as shown above the integral $I^{(-1)}_n$ is proportional
to the determinant of the matrix $M$ in \eqref{zero-mode} with $m=n-2$
and scaled parameters. However, the vector $v=(v_1,\ldots,v_n)$ belongs to
the kernel of $M$, $vM=0$, and, so, $\det M=0$.
For $n=2$, such a result follows also from \eqref{V-det} and the elliptic beta integral.
This statement proves the vanishing hypothesis of \cite{die-spi:selberg},
which was needed there for a proof of the evaluation formula for the
elliptic Dixon integral.

Applying Lemma \ref{lem-kernel} to relation \eqref{zero-mode}, we obtain
equality $\sum_{k=n}^{m+2}v_k d_k=0$, where $d_k=\det(M_{ll'})$ with
$l\in\{1,\ldots,n-1,k\}$, $l'\in\{1,\ldots,n\}$. The minors $d_k$
are proportional to the integrals
$I_n^{(m-n+1)}(t_1/q,\ldots,t_{n-1}/q,$ $t_n,\ldots,$ $t_k/q,
\ldots,$ $t_{m+2},t_{m+3}/p,$ $\ldots,t_{m+3+n}/p,t_{m+4+n},
\ldots,t_{2m+6})$. Substituting corresponding explicit expressions,
multiplying parameters $t_1,\ldots,t_{n-1}$ by $q$ and
$t_{m+3},\ldots,t_{m+3+n}$ by $p$, changing $m\to m+n-1$, and
permuting parameters appropriately, we obtain a recurrence for general
$I_n^{(m)}$-integrals.

\begin{thm}
If $t_1\cdots t_{2m+2n+4} = (pq)^{m+1}q$, then we have the following
$(m+2)$-term recurrence for $I^{(m)}_n$:
\begin{equation}
\sum_{1\le k\le m+2}
\frac{\prod_{m+3\le i\le 2n+2m+4} \theta_p(t_it_k/q)}
     {t_k\prod_{1\le i\le m+2;i\ne k} \theta_p(t_i/t_k)}
T_q(t_k)^{-1}
I^{(m)}_n(t_1,\dots,t_{2n+2m+4}) =0.
\label{rec-rel2}\end{equation}
\end{thm}

\begin{cor}\cite{rai:trans,spi:short}
If $t_1\cdots t_{2n+4}=pq$, then
\[
I^{(0)}_n(t_1,\dots,t_{2n+4}) = \prod_{1\le i<j\le 2n+4} \eg(t_it_j).
\]
\end{cor}

\begin{proof}
Indeed, both sides satisfy the same 2-term recurrence, and thus their ratio
is invariant under $T_q(t_i)^{-1}T_q(t_j)$, and similarly for $p$-shifts.
It follows that their ratio is independent of $t_1,\dots,t_{2n+4}$.
To determine the remaining factor, we may consider the limit of the ratio
as $t_{2n+3}t_{2n+4}\to 1$, and proceed by induction.
\end{proof}

\begin{cor}
If $t_1\cdots t_{2n+2m+4}=pq$, then we have the $\binom{n+m}{m}$-dimensional
determinant
\begin{align}
\det_{R,S\subset \{1,2,\dots,n+m\};|R|=|S|=m}
\left(
\prod_{r\in R} T_p(t_r) \prod_{s\in S} T_q(t_{s+n+m})
I^{(m)}_n(t_1,\dots,t_{2n+2m+4})
\right)&\notag\\
=
\bigl[
\prod_{1\le i<j\le n+m}
  t_j\theta_p(t_it_j^{\pm 1})
  t_{n+m+j}\theta_p(t_{n+m+i}t_{n+m+j}^{\pm 1})
\bigr]^{\binom{n+m-2}{m-1}}
&\notag\\  \times
\prod_{1\le i<j\le 2n+2m+4} \eg(t_it_j)^{\binom{n+m-1}{m}}.&
\notag
\end{align}
\end{cor}

\begin{proof}
Using the determinantal representation of $I^{(m)}_n$, we can express the
$(R,S)$ entry of the above determinant as
\begin{align}\nonumber
&\Bigl(
\prod_{\substack{i,j\in R^c\\i<j}}t_j\theta_p(t_it_j^{\pm 1})
\prod_{\substack{i,j\in S^c\\i<j}}t_{n+m+j}\theta_p(t_{n+m+i}t_{n+m+j}^{\pm
  1})
\Bigr)^{-1}\\
&\times
\det_{\substack{i\in R^c\\j\in S^c}}\Bigl(
 \prod_{\substack{1\le r\le n+m\\r\ne i}} T_p(t_r)
 \prod_{\substack{1\le s\le n+m\\s\ne j}} T_q(t_{n+m+s})
 I^{(m+n-1)}_1(t_1,\dots,t_{2n+2m+4})\Bigr),\notag
\end{align}
where $R^c=\{1,2,\dots,n+m\}\setminus R$.  The first two factors can be
pulled out of the determinant, as they are independent of the column or row
as appropriate; this gives an overall factor
\begin{align}\nonumber
\prod_{\substack{R\subset \{1,2,\dots,n+m\}\\|R|=m}}
\prod_{\substack{i,j\in R^c\\i<j}} (t_j\theta_p(t_it_j^{\pm 1}))^{-1}
&=
\prod_{1\le i<j\le n+m}
\prod_{\substack{R\subset \{1,2,\dots,n+m\}\\|R|=m;\ i,j\notin R}}
(t_j\theta_p(t_it_j^{\pm 1}))^{-1}\\
&=
\prod_{1\le i<j\le n+m}
(t_j\theta_p(t_it_j^{\pm 1}))^{-\binom{n+m-2}{m}}.
\nonumber\end{align}
We are thus left with computing a determinant of determinants.  These are
all minors of a fixed matrix, and thus our $\binom{n+m}{m}\times
\binom{n+m}{m}$ matrix is the $n$-th exterior power of the $n+m\times n+m$
matrix with $ij$ entry
\[
 T_p(t_i)^{-1} T_q(t_{n+m+j})^{-1}
 \prod_{1\le r\le n+m} T_p(t_r)T_q(t_{n+m+r})
 I^{(m+n-1)}_1(t_1,\dots,t_{2n+2m+4}),
\]
i.e., the matrix giving the action of the original matrix on the $n$-th
exterior power of its natural module.  In general, the $n$-th exterior
power of the $n+m\times n+m$ matrix $M$ has determinant
\[
\det(M)^{\binom{n+m-1}{n-1}},
\]
which can be seen easily by reduction to the diagonal case (the naturality
of the construction implies invariance under conjugation).  In our case,
the $n+m\times n+m$ matrix has determinant
\[
\left[
\prod_{1\le i<j\le n+m}t_j\theta_p(t_it_j^{\pm 1})
t_{n+m+j}\theta_p(t_{n+m+i}t_{n+m+j}^{\pm 1})
\right]
I_{m+n}^{(0)}(t_1,\dots,t_{2n+2m+4}),
\]
which implies the desired result.
\end{proof}

{\em Remark 1.}  When $m=1$, this is essentially an elliptic version of Varchenko's
determinant of univariate hypergeometric integrals \cite{var1,rich}.
Indeed, Varchenko's determinant is equivalent to Dixon's integral
evaluation, which is a limit of the $I^{(0)}_n$ evaluation formula
\cite[Thm. 7.2]{rai:limits}.  Convergence of the matrix itself is
somewhat more subtle, as the domain of integration tends to pass through
algebraic singularities of the integrand in the limit, thus giving rise to
somewhat tricky phase issues.  One does find, however, that
\begin{align}
\lim_{q\to 1^-}
&\frac{\eg(q^{\alpha_1+\alpha_2})}{\eg(q^{\alpha_1},q^{\alpha_2},q^{\alpha^+_1+\alpha^+_2}a_1a_2,q^{\alpha^+_1+\alpha^-_2}a_1/a_2,q^{\alpha^-_1+\alpha^+_2}a_2/a_1,q^{\alpha^-_1+\alpha^-_2}/a_1a_2)}\notag\\
&
I^{(n-1)}_1(pq^{1-\beta^+_1}/b_1,\dots,pq^{1-\beta^+_n}/b_n,
            q^{-\beta^-_1}b_1,\dots,q^{-\beta^-_n}b_n,\notag\\
&\phantom{I^{(n-1)}_1(}
            q^{\alpha^+_1}a_1,q^{\alpha^-_1}/a_1,
            q^{\alpha^+_2}a_2,q^{\alpha^-_2}/a_2;p,q)\notag\\
&{}=
|\theta_p(a_1 a_2^{\pm 1})|^{1-\alpha_1-\alpha_2}
\frac{\Gamma(\alpha_1+\alpha_2)}{\Gamma(\alpha_1)\Gamma(\alpha_2)}
(2\pi(p;p)_\infty^2)\notag\\
&\phantom{{}={}}
\int_{z\in [a_1,a_2]}
|\theta_p(a_1 z^{\pm 1})|^{\alpha_1-1}
|\theta_p(a_2 z^{\pm 1})|^{\alpha_2-1}
\prod_{1\le r\le n} \theta_p(b_r z^{\pm 1})^{\beta_r}
\frac{|\theta_p(z^2)|dz}{2\pi\sqrt{-1}z},
\end{align}
where $a_1$, $a_2$, $b_1$,\dots,$b_n$ are on the unit circle with positive
imaginary part and $\Re(a_1)>\Re(a_2)$, the exponents satisfy the
convergence conditions $\Re(\alpha^\pm_i)>0>\Re(\beta^-_i)$, and one has
the balancing condition
$
\alpha_1+\alpha_2 = \sum_i \beta_i,
$
with $\alpha_i:=\alpha^+_i+\alpha^-_i$, $\beta_i:=\beta^+_i+\beta^-_i$.
The domain of integration is the counterclockwise arc from $a_1$ to $a_2$.
The proof is as in Theorem 7.2 of \cite{rai:limits}; note also that by
Lemma 7.1, op. cit., one has
\[
\theta_p(b_r z^{\pm 1})^{\beta_r}
=
\exp\bigl(-\beta_r(\log(-b_r z)+\log(-b_r/z))/2\bigr)
|\theta_p(b_r z^{\pm 1})|^{\beta_r}.
\]
The resulting factor is in particular locally constant in $z$ on the upper
semicircle, changing only when $z$ passes over one of the singularities
$b_i$.  Also, as noted in \cite{rai:limits}, the change of variables
$x=-\theta_p(z)^2/\theta_p(-z)^2$ turns this into a higher-order ordinary
beta integral, precisely as appears in Varchenko's determinant identity.
Our result can also be viewed as a generalization of the main result of
the work of Aomoto and Ito \cite{ai3},
which corresponds to a further degeneration of the
trigonometric integral of \cite[Thm. 5.6]{rai:limits}; the Jackson integral
is obtained as the sum of residues obtained upon shrinking the contour to
0.

The integrals $I_n^{(m)}$ thus provide solutions of the system of
recurrence relations \eqref{rec-rel1} and  \eqref{rec-rel2}
and their partners obtained after permutations of parameters.
However, these recurrence relations do not imply the constraint
$|q|<1$ which is needed for the definition of $I_n^{(m)}$.
The construction of solutions of the elliptic hypergeometric
equation with $|q|>1$ \eqref{q>1} extends to arbitrary values of $n$ and $m$.
Modulo some ellipticity factor, the $q$-shift equations satisfied by
$$
I^{(m)}_n(t_1,\dots,t_{2n+2m+4})
\prod_{1\le i\le 2n+2m+4} \eg(t_i z^{\pm 1})^{-1}
$$
are invariant under the transformation
$
t_i\mapsto p^{a_i}/t_i,$ $ q\mapsto 1/q,
$
where $a_1,\dots,a_{2n+2m+4}$ is any sequence of integers or half-integers
(but not mixed) such that $\sum_{i=1}^{2n+2m+4}a_i = 2m+2$. That is, for
both recurrence relations, the rescaled coefficient of each term gets
multiplied by the same quantity under such a transformation.  This observation
provides us with the solutions of those recurrences for $|q|>1$.
Solutions with $|q|=1$ are obtained (as in the $n=m=1$ case)
by using the modified elliptic gamma-function which we do not consider
for brevity.

\section{A proof of the transformation formula}

The two recurrences we have given for $I^{(m)}_n$ are also
recurrences for the right-hand side of relation \eqref{trafo}: the transformation
simply swaps the two kinds of recurrence. The proof of Theorem \ref{tr-thm} using
these difference equations boils down to showing that $I^{(m)}_n$ is the
unique meromorphic solution of the full system of $p$- and $q$-difference
equations, up to an overall constant (which is then easy to obtain via a
limit). As a subcase, this proves Theorem \ref{V-char} as well.
  It is not particularly elementary; it involves some
difference/differential Galois theory \cite{vanderPut/Singer}.  We start by
giving the key auxiliary statement needed for us.

One technical issue that arises in our application of difference Galois
theory is that the literature, and many of the main statements, require
that the constant field be algebraically closed.  Since we are dealing in
our case with a {\em family} of difference equations, this constraint is
too strict.  However, we can evade this issue via a suitable base change,
as follows.  Let $(k,\tau)$ be a difference field (i.e., $\tau$ is an
automorphism of the field $k$) of characteristic 0 such that the field
$k^\tau$ of constants is algebraically closed in $k$.  (In other words, the
only finite orbits of $\tau$ on $k$ are of length 1.)  There is then a
canonical extension of $\tau$ to the field $l:=k\otimes_{k^\tau}
\overline{k^\tau}$, such that {\em the algebraic closure}
$\overline{k^\tau}$ is the new constant
field.  We thus obtain a well-defined notion of Galois group over $l$.

\begin{lem} Let $A\in GL_n(k)$, and let $W$ be the space of vectors $w\in k^n$
such
that $\tau(w)=A w$.  Let $G$ be the Galois group of this difference
equation
(viewed as an equation with coefficients in $l$), with associated
representation $V$.  Then
\[
\dim_{k^\tau}(W) = \dim_l(V^G).
\]
\end{lem}

\begin{proof}
The absolute (ordinary) Galois group ${\it Gal}(k^\tau)$ acts naturally on $l$,
and commutes with the action of $\tau$; in particular, we have a natural
isomorphism ${\it Gal}(k^\tau)={\it Gal}(l/k)$, and the latter is well-defined.  In
particular, if we extend coefficients of the difference equation to $l$,
the resulting vector space is stable under ${\it Gal}(l/k)$, and thus admits a
basis of $k$-rational vectors.  In other words, the dimension
$\dim_{k^\tau}(W)$ is unchanged under this coefficient extension, and we
may therefore assume $k=l$, or in other words that the field of constants
in $k$ is algebraically closed.

Now, let the $k$-algebra $R$ be a Picard-Vessiot ring of the difference
equation.  Then, by definition, the set of solutions $w\in R^n$ of the
difference equation is an $n$-dimensional vector space over $k^\tau$ with a
$G$ action equivalent to the representation $V$.  A solution has
coefficients in $k$ iff it is invariant under the action of $G$, and thus
the result follows.
\end{proof}

{\em Remark 2.}  A similar lemma holds for the differential case (replacing
$GL_n$ by its Lie algebra); note that in that case, the constant field is
automatically algebraically closed in the coefficient field.

{\em Remark 3.}  More generally, if $\rho$ is any rational representation of
$GL_n(k)$, one can apply the lemma to the difference equation
$\tau(w)=\rho(A) w$, using the fact that the new Galois group is simply
$\rho(G)$.

Applying remark 3 to the adjoint representation gives the following:

\begin{cor}
Let $A,G$ be as above, and let $W$ be the space of matrices $M\in {\it End}(k^n)$
such that $\tau(M)=A M A^{-1}$.  Then
\[
\dim_{k^\tau}(W) = \dim_l({\it End}_G(V)).
\]
In particular, if $G$ is irreducible, then $W$ is $1$-dimensional, and
(since $I\in W$) $W$ consists of scalar matrices.
\end{cor}

That ${\it End}_G(V)=l$ when $V$ is irreducible is a standard fact (Schur's
lemma) of representation theory. Any endomorphism which is not a
multiple of the identity has at least one proper eigenspace (since $l$
is algebraically closed), and each eigenspace is invariant, making the
representation reducible.

We can pass now to the proof of our transformation formula itself.
In this case the field $k$ coincides with the field of $p$-elliptic
functions in one of the parameters $x$ (actually, in all parameters),
and $\tau$ is the $q$-shift operator, $\tau(f(x))=f(qx)$. Therefore
$k^\tau$ is a field of constants independent on $x$ (the only
simultaneously $p$- and $q$-elliptic set of  $x$-functions). In
order to apply the above theory, we need to show three things: first, that
our two recurrences can be combined to give a difference equation in matrix
form as above, second, that the coefficients of this difference equation can
be made elliptic upon suitable renormalization, and third, that the
resulting elliptic difference equation has irreducible Galois group.

\begin{proof}
To see that the recurrences combine to give a matrix difference equation of
order $\binom{n+m}{m}$, we need to show that one can use the recurrences to
express
\[
I^{(m)}_n(qt_1,\dots,qt_m,t_{m+1},\dots,qt_{2m+2n+3},t_{2m+2n+4}/q;p,q)
\]
as a linear combination of
\[
\prod_{i\in S} T_q(t_i)
I^{(m)}_n(t_1,\dots,t_{2m+2n+4};p,q),
\]
where $S$ ranges over $m$-element subsets of $\{1,2,\dots,n+m\}$.
Using the $n+2$-term recurrence, we obtain a linear dependence between the
original integral, the integral
\[
I^{(m)}_n(qt_1,\dots,qt_m,t_{m+1},\dots,t_{2m+2n+3},t_{2m+2n+4};p,q),
\]
and the $n$ integrals
\[
T_q(t_i)
I^{(m)}_n(qt_1,\dots,qt_m,t_{m+1},\dots,t_{2m+2n+3},t_{2m+2n+4}/q;p,q)
\]
for $m+1\le i\le m+n$.  By symmetry, it suffices to consider the term with
$i=m+1$.  But then the $m+2$ term recurrence gives a linear dependence
between the $m+2$ integrals
\[
T_q(t_i)^{-1}
I^{(m)}_n(qt_1,\dots,qt_{m+1},\dots,t_{2m+2n+3},t_{2m+2n+4};p,q)
\]
for $i\in \{1,\dots,m+1,2m+2n+4\}$.  We thus obtain a matrix difference
equation of the form required.  (One also notes that the corresponding
matrix $A$ is quite sparse; the entry corresponding to a pair $S$, $T$ of
$m$-subsets of $\{1,2,\dots,n+m\}$ is 0 unless $S\cap T\ge m-1$.)

Next, for ellipticity, we consider the renormalization
\[
\left(
\frac{\eg(v^{\pm 2})}
     {\prod_{1\le r\le 2m+2n+4} \eg(t_r v^{\pm 1})}
\right)^n
I^{(m)}_n(t_1,\dots,t_{2m+2n+4};p,q).
\]
This differs from the corresponding minor of the matrix for $I^{(m+n-1)}_1$
by multiplication by a pair of factors, one of which is $p$-elliptic in all
variables other than $v$, and the other of which is similarly
$q$-elliptic.  We thus find that the matrix $A$ is essentially just the
$n$-th exterior power of the corresponding matrix for $I^{(m+n-1)}_1$, up
to a pair of diagonal matrices that combined have no effect on
ellipticity.  It will thus suffice to show that the difference equation is
elliptic when $n=1$.

Consider the $m+1\times m+1$ matrix
\begin{align}
M(x)_{ij}
:=
T_p(t_i)^{-1}T_q(t_j)^{-1}&
\frac{\eg(v^{\pm 2})
      \prod_{1\le r\le m+1} \eg(v^{\pm 1}/t_r)}
     {\eg(x v^{\pm 1},v^{\pm 1}/Tx)
      \prod_{m+2\le r\le 2m+4} \eg(t_r v^{\pm 1})}\notag\\
&\times I^{(m)}_1(pqt_1,\dots,pqt_{m+1},t_{m+2},\dots,t_{2m+4},x,1/Tx;p,q),\notag
\end{align}
where $T=\prod_{1\le r\le 2m+4}t_r$.  This is symmetrical between $p$ and $q$
(being replaced by its transpose when $p$ and $q$ are swapped), so it
suffices to consider its behavior under a $q$ shift.  One finds, in fact
(using the two recurrences as described above), that
\[
M(qx) = A(x)M(x),
\]
where
\begin{align}
A(x)_{ij}
&{}=
\delta_{ij}
\frac{\theta_p(xt_i^{\pm 1},Tx v^{\pm 1})}
     {\theta_p(xv^{\pm 1},Tx t_i^{\pm 1})}
+
\frac{\theta_p(t_i v^{\pm 1},Tx v^{\pm 1})}
     {\theta_p(t_i (Tx)^{\pm 1},xv^{\pm 1}))}
\notag\\
&{} \times
\frac{\theta_p(Tx^2)\prod_{1\le l\le m+1}\theta_p(t_l Tx)}
{\prod_{m+2\le l\le 2m+4} \theta_p(Tx/t_l)}
\frac{\theta_p(t_j x)\prod_{m+2\le l\le 2m+4} \theta_p(1/t_lt_j)}
     {\theta_p(t_j Tx,v^{\pm 1}/t_j)\prod_{1\le l\le m+1;l\ne j}\theta_p(t_l/t_j)}\notag
\end{align}
These coefficients are readily verified to be elliptic as required.

It remains only to prove that the Galois group is (generically) irreducible
for all $m$, $n$.  Now, the Galois group for $n>1$ is the $n$-th exterior
power of the group for $n=1$, so it will suffice to show that the generic
group for $n=1$ contains $SL_{m+1}$ (all nonzero exterior powers of which
are irreducible).  We can proceed as in the proof of Theorem 3.3.3.1 in
\cite{andre}.  The point is that, by Andr\'e's theory, the Galois
group can only become smaller under specialization, including degeneration
to a differential equation.  We have already discussed the fact that the integral
$I^{(m)}_1$ can be degenerated to a higher-order classical beta integral,
and indeed one can obtain a basis of the corresponding differential
equation in that way.  It thus follows that the elliptic hypergeometric
difference equation degenerates under that limit to the Jordan-Pochhammer
differential equation.  But this is known \cite{Takano/Bannai} to have
generic Galois group containing $SL_{m+1}$.
\end{proof}

{\em Remark 4.}  Note, in particular, that the given formula for $A$ indeed
converges to the identity matrix in the Jordan-Pochhammer limit, as in
particular $T\to 1$, making the diagonal contribution converge to 1; the
off-diagonal contribution vanishes since $t_jt_{j+m+1}\to 1$.  One can
moreover directly compute the limiting differential equation, and use the
rigidity of the Jordan-Pochhammer equation to verify that the two equations
are equivalent.

\medskip

E.M.R. is supported in part by the National Science Foundation, grant
DMS0401387. V.S. is supported in part by the Russian Foundation for
Basic Research (RFBR), grant  08-01-00392, and by the Max Planck Institute for
Mathematics (Bonn) during the visit of which a part of this work was done.

\end{document}